\documentclass[12pt]{amsart} 
\usepackage{euler}
\usepackage{euscript}
\usepackage{multicol} 
\usepackage{amssymb}

\usepackage{color}
\usepackage{xcolor}

\definecolor{amaranth}{rgb}{0.9, 0.17, 0.31}
\definecolor{ao}{rgb}{0.0, 0.0, 1.0}
\definecolor{ao(english)}{rgb}{0.0, 0.5, 0.0}
\definecolor{deepmagenta}{rgb}{0.8, 0.0, 0.8}

\usepackage{enumitem}

\usepackage[pdftex]{graphicx}
\usepackage{amsmath} 
\usepackage{times} 
 
\def\XXint#1#2#3{{\setbox0=\hbox{$#1{#2#3}{\int}$}
     \vcenter{\hbox{$#2#3$}}\kern-.5\wd0}}

\theoremstyle{plain}
\newtheorem{theorem}{Theorem}[section]
\newtheorem{proposition}{Proposition}[section]
\newtheorem{corollary}{Corollary}[section]
\newtheorem{lemma}{Lemma}[section]
\newtheorem{remark}{\bf Remark}[section]
\theoremstyle{definition}

\newcommand{\sgn}{\mathop\mathrm{sgn}}

\newcommand{\Tr}{\mathop\mathrm{Tr}}

\usepackage[pdftex]{graphicx}
\usepackage{amsmath} 
\usepackage{times} 
 
\def\XXint#1#2#3{{\setbox0=\hbox{$#1{#2#3}{\int}$}
     \vcenter{\hbox{$#2#3$}}\kern-.5\wd0}}

\begin{document}

\title[Szeg\"o limit theorem and Heisenberg Laplacian]
{Szeg\"o limit theorem and Heisenberg Laplacian
}

\author{Jiyu Fan}
\address{Jiyu Fan:  Department of Mathematics, Imperial College London,
j.fan23@imperial.ac.uk}

\author{Ari Laptev}
\address{Ari Laptev:  Department of Mathematics, Imperial College London,
a.laptev@imperial.ac.uk}

\begin{abstract} The aim of this paper is to obtain a version of the classical Szeg\"o limit theorem, where instead of the operator of second derivative on a circle we consider the Heisenberg-H\"ormander Laplacian in $L^2(\Bbb R^3)$. Besides, we derive a sharp inequality for convex functions that we call Szeg\"o-type inequality.
\end{abstract}



\maketitle
 
 \setcounter{equation}{0}

 \section{Introduction} 
 
Let us consider the selfadjoint operator $T=(-d^2/dt^2)^{1/2}$ in $L^2(\Bbb S)$, where $\Bbb S$ is the unit circle. The spectrum of the operator $T$ is discrete and its eigenfunctions and eigenvalues are
$\{e^{ik t}\}$ and $\{|k|\}$ respectively, where $k\in\Bbb Z$. Denote by $P_n$, $n\in\Bbb N$, the orthogonal projection in $L^2(\Bbb S)$ onto the subspace spanned by $e^{ikt}$ with $|k| \le n$. Assume that $B$ is the multiplication operator by a smooth function $b>0$ in $L^2(\Bbb S)$.
The classical Szeg\"o limit theorem from 1920 states that under some additional assumptions on $b$, see \cite{Sz0,Sz1,GS},
\begin{equation}\label{Sz1} 
{\rm Tr} \, \log (P_n B P_n) = \frac{n}{\pi} \int_0^{2\pi} \log b(t) \, dt + O(1), \quad {\rm as} \quad n\to\infty.
\end{equation} 
In 1952 \cite{Sz2} Szeg\"o obtained a remarkable result, where he was able to obtain a sharp second term in \eqref{Sz1} under the condition that $b>0$ and its derivative $b'$ satisfies a Lipschitz condition with exponent $\alpha$, $0<\alpha\le1$. A continuous analogue of Szeg\"o's result was obtained by M. Kac in 1954 \cite{K} and this was 
extended to arbitrary dimension by H. Widom in 1960, \cite{W1}. 

\medskip
\noindent
During the last six decades many different versions of Szeg\"o type theorems appeared in the literature, where it was pointed out that such results could be interpreted as a version of  Weyl asymptotic formulae for differential and pseudodifferential operators. Among them are papers \cite{GO, GLZ, J, JZ, LRS, O, R, S1,S2,W2,Z}. 
Some results were obtained for various classes of differential and pseudodifferential operators, where the asymptotic formulae were proved without any remainder estimates \cite{H, R, Z}. More
precise results have been obtained only in some special cases  \cite{GO, O} and \cite{J}. One important extension of Szeg\"o type limit theorem was considered by 
H. Widom in \cite{W0}, where the author found an asymptotic formula for the trace of the function of a truncated Wiener-Hopf type operator in dimension one and conjectured a multi-dimensional generalization of that formula for a pseudo-differential operator with a symbol having jump discontinuities in both variables.
In 1990 he proved this conjecture for the special case when the jump in any of the two variables occurs on a hyperplane. The multi-dimensional version of this conjecture was fully proved by A.V Sobolev in \cite{PS,S1,S2}.

\medskip
\noindent
In \cite{LapSaf}, one of the authors of this paper and Yu. Safarov obtained some inequalities for selfadjoint operators in a Hilbert space and applied them to various classes of Szeg\"o type limit theorems. 
The aim of this paper is to apply this technique to the case where the projection operators appearing in \eqref{Sz1}  are related to the spectrum of the Heisenberg-H\"ormander Laplacian. 

\medskip
\noindent
Let $d= 3$ and denote 
\begin{equation*}
X_1: = i \,\partial_{x_1}+\frac{i}{2}x_2\, \partial_{x_3}, \quad 
X_2:= i \,\partial_{x_2} - \frac{i}{2} x_1\, \partial_{x_3}, \quad X_3 = -i \partial_{x_3}.
\end{equation*} 
The Lie vector fields $X_1$, $X_2$ and $X_3$ are infinitesimal generators
of the first Heisenberg group which can be described as the set 
$\mathbb R^2\times \mathbb R$ equipped with the group law
\begin{equation*}
  (x_1,x_2,x_3)(y_1,y_2,y_3) = \left(x_1+y_1,\ x_2+y_2,\ x_3+y_3 +
  \frac{x_1y_2-x_2y_1}2\right ).
\end{equation*}
These vector fields satisfy the canonical commutation relation
$$  
[X_1, X_2] =  i X_3.
$$

Let us introduce the quadratic form defined on $C_0^\infty(\Bbb R^3)$ 
\begin{equation}\label{a}
a[u] := \int_{\mathbb R^3} \Big(|X_1 u(x)|^2 + |X_2 u(x)|^2\Big)\, dx.
\end{equation}
Its Friedrichs extension defines a self-adjoint operator $A$ in $L^2(\mathbb R^3)$
\begin{equation}\label{A}
A := X_1^2 + X_2^2.
\end{equation}
The operator $A$ is a subelliptic differential operator that is often called the Heisenberg-H\"ormander Laplacian, \cite{H}.
Denoting by $\mathcal F_3$ the partial Fourier transform 
$$
\mathcal F_3 u(x',\xi_3)= (2\pi)^{-1/2} \int_{-\infty}^\infty
e^{-ix_3\xi_3} u(x',x_3) \,
dx_3, \qquad x'=(x_1,x_2)
$$
we have
\begin{equation}\label{FAF*}
\mathcal F_3 A \mathcal F_3^{*} 
= \left(i\partial_{x_1}-\frac12 x_2\xi_3\right)^2
+\left(i\partial_{x_2}+\frac12 x_1\xi_3\right)^2
=(i\nabla_{x'}+\xi_3\mathbf A(x'))^2,
\end{equation}
where $\mathbf A(x')=\frac12 (-x_2,x_1)$. This operator can be identified with a Laplacian in $x'$
with a constant magnetic field of strength $|\xi_3|$. Its spectrum consists of the Landau levels
\begin{equation*}
  \mu_n(\xi_3):=|\xi_3|(2n+1),
\qquad n \in \mathbb N_0=\{0,1,2,\ldots\}.
\end{equation*}
Therefore the spectrum of the operator $A$ covers the interval $[0,\infty)$.
Note that $\mathcal F_3 A\mathcal F_3^*$ acts as a multiplication
operator with respect to the variable $\xi_3$. By the spectral
theorem, 
\begin{equation*}
\mathcal F_3A\mathcal F_3^*=\sum_{n=0}^\infty \mu_n(\xi_3)
\Pi_{\xi_3,n},
\qquad \Pi_{\xi_3,n}=\Pi'_{\xi_3,n}\otimes I_{L^2(\mathbb R)},
\end{equation*}
where the projector $\Pi'_{\xi_3,n}$ has an explicit representation as a kernel of an integral operator.
It is well known that this kernel equals 
\begin{equation}\label{ProjPrm}
\Pi'_{\xi_3,n}(x',y') = \frac{|\xi_3|}{2\pi} e^{-i\xi_3 x' \times y'/2 - |\xi_3| |x'-y'|^2/4} \, L_{n} \left(|\xi_3| |x'-y'|^2/2\right),
\end{equation}
where $L_{n} $ are Laguerre polynomials. It is well known that the value of the kernel $\Pi'_{\xi_3,n}(x',y')$ on the diagonal is constant and equals 
\begin{equation} \label{projker1}
  \Pi'_{\xi_3,n}(x',x')=\frac{|\xi_3|}{2\pi}.
\end{equation}
see, e.g., \cite{hlw}, \cite{FrLap}, \cite{RW}. 

\medskip
\noindent
For any $\lambda >0$ we consider the spectral projection ${P}_\lambda$ of the operator $A$. Its kernel is identified with the spectral function $e(x,y,\lambda) $ that is given by the formula
\begin{equation} \label{SpF}
e(x,y,\lambda) =(2\pi)^{-1} \sum_n \underset{|\xi_3|(2n+1)<\lambda}\int \Pi_{\xi_3,n}(x',y') e^{i\xi_3(x_3-y_3)} d\xi_3.
\end{equation}
The first aim of this note is the following sharp inequality that we call Szeg\"o  type inequality

\begin{theorem}\label{SBLY}
Let $B$ be the operator of multiplication by a real function $b$ in $L^2(\mathbb R^3)$. Assume that $\varphi$ is a convex function on $\mathbb R$ such that $\varphi\ge0$, $\varphi(0) =0$
and such that 
$$
\int_{\mathbb R^3} \varphi (b(x))\, dx < \infty.
$$
Then the operator $P_\lambda\varphi(B) P_\lambda \in \mathfrak{S}_1$ is trace class and 
\begin{equation}\label{S_ineq}
{\rm Tr}\, \varphi (P_\lambda BP_\lambda)  \le \frac{\lambda^2}{32} \, \int_{\mathbb R^3} \varphi (b(x)) dx. 
\end{equation}
The constant in the right hand side of the inequality \eqref{S_ineq} is sharp.
\end{theorem}

\noindent
Note that here we do not assume any boundedness or smoothness assumptions on $b$ that are usually imposed in Szeg\"o limit theorems. 

\smallskip
\noindent
The proof of this theorem is somewhat similar to the proof of a sharp Li$\&$Yau type inequality for Heisenberg Laplacian in a domain $\Omega\subset\Bbb R^3$ of finite measure with Dirichlet boundary conditions obtained in \cite{HL}. For some discussions on Li$\&$Yau type inequalities see \cite{ELV, L} and
\cite{FLW}.

\smallskip
\noindent
The sharpness of the constant in the inequality \eqref{S_ineq} is confirmed by the following Szeg\"o type limit theorem that is our second main result.

\medskip
\begin{theorem}\label{Szego}
Let $B$ be the multiplication by a real function $b(x)\in C_0^2(\mathbb R^3)$ in $L^2(\mathbb R^3)$ and let 
$\psi\in W^{2,\infty}(\mathbb R)$ with $\psi(0) = 0$. Then
\begin{multline}\label{SLT}
{\rm Tr}\,  P_\lambda\psi(P_\lambda B P_\lambda) P_\lambda = 
\int_{\mathbb{R}^3} \psi(b(x)) \, e(x,x,\lambda) \, dx + O(\lambda^{3/2})\\
=
\frac{\lambda^2}{32} \, \int_{\mathbb{R}^3} \psi (b(x)) dx + O(\lambda^{3/2}), \quad {\rm as} \quad \lambda\to\infty.
\end{multline}
\end{theorem}


\begin{remark}
The remainder term $O(\lambda^{3/2})$ in the asymptotic formula \eqref{SLT} is likely to be sharp.
\end{remark}

\medskip
\noindent
The inequality \eqref{S_ineq} is sharp and uniform in $\lambda$. It implies the following statement:
\begin{corollary}
If in Theorem \ref{Szego} $\psi$ is a convex function satisfying the conditions in Theorem \ref{SBLY}, then the remainder term in \eqref{SLT} is uniformly negative.
\end{corollary}


\section{Some auxiliary results}

\smallskip
\noindent
In this section we present some results that are used for the proofs of the main theorems. Many of them are known and we reprove them for the sake of completeness. The following statement was obtained in \cite{FrLap}
\begin{proposition}\label{prop0} 
Let $\Omega\subset\mathbb R^2$ be a set of finite measure and let $\Pi'_{\xi_3,n}(x',y')$ be the projections defined in
\eqref{ProjPrm}. Then
\begin{equation} \label{projker2}
\int_\Omega  \int_{\mathbb R^2}  \left|\Pi'_{\xi_3,n}(x',y')\right|^2 dx' dy' = \frac{|\xi_3| }{2\pi}  \, |\Omega|.
\end{equation}
\end{proposition}
\begin{proof} 
Note that since the operators $\Pi'_{\xi_3,n}$ are orthogonal projections in $L^2(\mathbb R^2)$ we have 
$\Pi'_{\xi_3,n} = (\Pi'_{\xi_3,n})^* $ and $(\Pi'_{\xi_3,n})^2 = \Pi'_{\xi_3,n}$. Then using \eqref{projker1} we find
\begin{multline*}
\int_\Omega \int_{\mathbb R^2}  \left|\Pi'_{\xi_3,n}(y',x')\right|^2 dx' dy' = 
\int_\Omega  \int_{\mathbb R^2} \Pi'_{\xi_3,n}(y',x') \overline{\Pi'_{\xi_3,n}(x',y')} \,dx' dy' \\
= \int_\Omega \Pi'_{\xi_3,n}(y',y') \, dy'
= \frac{|\xi_3|}{2\pi}  \, |\Omega|.
\end{multline*}
\end{proof}

\begin{proposition}\label{prop1} 
Let $B$ be the operator of multiplication by a function $b(x)$ in $L^2(\mathbb R^3)$.
Then, for any real function $\psi$ on $\mathbb R$ such that $\int_{\mathbb R^3} \psi(b(x)) \, dx<\infty$, we have
$$
{\rm Tr}\, P_\lambda \psi (B) P_\lambda = \frac{\lambda^2}{32} \, \int_{\mathbb {R}^3} \psi (b(x)) dx,
$$
where the operator $P_\lambda$ is the spectral projection of $A$.
\end{proposition}

\begin{proof} Indeed, using \eqref{projker1} and \eqref{SpF} we find
\begin{multline*}
{\rm Tr}\, P_\lambda \psi (B) P_\lambda 
= 
\frac{1}{2\pi} \, \sum_n \int_{\mathbb R^3} \underset{|\xi_3|(2n +1) <\lambda}
\int \psi(b(x)) \,  \Pi_{\xi_3,n}(x',x') \, d\xi_3 dx\\
= 
\frac{1}{4\pi^2} \, \sum_n \int_{\mathbb R^3} \underset{|\xi_3|(2n +1) <\lambda} \int \psi(b(x)) \,  |\xi_3| \,d\xi_3 dx.
\end{multline*}
Changing variables we arrive at 
\begin{multline*}
{\rm Tr}\, P_\lambda \psi (B) P_\lambda  =  \frac{1}{2\pi^2} \, \sum_n \frac{1}{(2n+1)^2}\, \int_{\mathbb R^3} \psi (b(x)) dx 
\int_0^\lambda  \eta \, d\eta \\
= 
 \frac{\lambda^2}{4\pi^2} \, \sum_n \frac{1}{(2n+1)^2}\, \int_{\mathbb R^3} \psi (b(x)) dx = 
  \frac{\lambda^2}{4\pi^2} \, \frac{\pi^2}{8} \, \int_{\mathbb R^3} \psi (b(x)) dx. 
\end{multline*}
\end{proof}

\medskip
\noindent
Let $\mathfrak{S}_1 \subset \mathfrak{S}_2$ be the trace and the Hilbert-Schmidt class of operators, respectively in a Hilbert space $\mathcal H$, $P$ be an orthogonal projection in $\mathcal H$ and let $B$ be a selfadjoint operator in $\mathcal H$ whose spectrum we denote by 
$\sigma(B)$. Let us introduce the set $\mathcal{I}$ such that 
$$
\mathcal{I} = \underset{0\le t\le1}\cup  t\,\sigma(B)\ \subset\ \mathbb R,
$$ 
where $\sigma(B)$ is the spectrum of $B$. Then the spectra of operators $B$ and $PBP$ lie in 
$\mathcal{I}$.

\noindent
We now recall the well-known Berezin-Lieb inequality for convex functions (see
\cite{Bz1}, \cite{Bz2}, \cite{Lieb}). The most general case  of this inequality was proved in \cite{LapSaf}.

\begin{proposition}\label{B-L}
Let $\varphi\ge0$ be a convex function on $\mathcal{I}$ with $\varphi(0) = 0$ if $0\in\mathcal{I}$. Assume that $PB$ is a compact operator in  $\mathcal H$ and 
$P\varphi(B)P  \in \mathfrak{S}_1$.
Then 
$$
{\rm Tr}\, P\varphi(B)P \ge {\rm Tr}\, P\varphi(PBP) P.
$$
\end{proposition}

\noindent
Here we give the proof of this statement for completeness. 

\begin{proof}
Let $\eta_k$ be an orthonormal basis in $P\mathcal H$ formed by the eigenvectors $\eta_k$  of
the compact selfadjoint operator $PBP$. Denote by $E_B$ the spectral measure of the operator $B$. 
Then 
\begin{multline*}
(P\varphi(PBP) P \eta_k, \eta_k) = \left(\varphi(PBP)  \eta_k, \eta_k\right) = \varphi\left((PBP \eta_k,\eta_k)\right)
\\
= \varphi\left((B \eta_k,\eta_k)\right) = \varphi\left(\int \mu\, \left( d E_B(\mu) \eta_k, \eta_k\right) \right).
\end{multline*}
Since $\|\eta_k\| = 1$ we find
$$
\int  d(E_B(\mu) \eta_k, \eta_k) = \|\eta_k\|^2 = 1
$$ 
and using Jensen's inequality we conclude
\begin{multline*}
\left(P\varphi(B)P\eta_k,\eta_k\right) - \left(P\varphi(PBP) P \eta_k,\eta_k\right) 
= \left(\varphi(B)\eta_k,\eta_k\right) - \left(\varphi(PBP)  \eta_k,\eta_k\right) \\
=
\int \varphi(\mu)\, \left( d E_B(\mu) \eta_k, \eta_k\right) - \varphi\left(\int \mu\, \left( d E_B(\mu) \eta_k, \eta_k\right)\right) \ge0.
\end{multline*}
Summing these inequalities we complete the proof.
\end{proof}

\medskip
\noindent
The inequality obtained in Proposition \ref{B-L} can be extended to a two-sided estimate.
From the Berezin-Lieb inequality one can deduce the following result (see \cite{LapSaf}).

\begin{theorem}\label{DSide}
Let $PB\in\mathfrak{S}_2$. Then for any
function $\psi\in W^{2,\infty}(\mathcal{I})$
$$
P\psi(B)P-P\psi(PBP)P\in\mathfrak{S}_1
$$
and
\begin{equation}\label{DoubSide}
\left|\,\Tr\left(P\psi(B)P-P\psi(PBP)P\right)\,\right|
\ \le\ \frac 12\,\|\psi''\|_{L^\infty(\mathcal{I})}\,\|PB(I-P)\|_{\mathfrak{S}_2}^2\,.
\end{equation}
\end{theorem}

\begin{proof}
Set
\[
D_\psi := P\psi(B)P-P\psi(PBP)P.
\]
If $\ell(t)=\alpha t+\beta$ is affine, then $D_\ell=0$. Hence we may replace $\psi$ by $\psi-\psi(0)-\psi'(0)t$ and assume from now on that
\[
\psi(0)=\psi'(0)=0.
\]
Let
\[
q(t)=\frac{t^2}{2}, \qquad M=\|\psi''\|_{L^\infty(\mathcal{I})}.
\]
By Taylor's formula,
\[
|\psi(t)|=\left|\int_0^t(t-s)\psi''(s)ds \right|\le  Mq(t), \qquad t\in \mathcal{I}.
\]
Since $B$ is self-adjoint, this pointwise estimate can be transferred to an operator estimate. Applying $P$ we get 
$$-MPq(B)P\le P\psi(B)P\leq MPq(B)P,$$
similarly for $Pq(PBP)P$.
Since $PB\in\mathfrak{S}_2$, one has
\[
Pq(B)P=\frac12 PB^2P\in\mathfrak{S}_1,
\qquad
q(PBP)=\frac12 PBPBP\in\mathfrak{S}_1.
\]
Therefore $P\psi(B)P$ and $P\psi(PBP)P$ are also trace class.

\noindent
Now define
\[
\varphi_\pm(t)=Mq(t)\pm \psi(t).
\]
Because $\varphi_\pm''(t)=M\pm\psi''(t)\ge0$, both functions are convex on $\mathcal{I}$. After extending $\varphi_\pm(t)$ convexly from $\mathcal{I}$ to $\mathbb{R}$, we apply the Berezin--Lieb inequality to $\varphi_+$ and $\varphi_-$ and obtain
\[
0\le \Tr D_{\varphi_\pm}
= M\Tr D_q \pm \Tr D_\psi.
\]
Hence
\[
|\Tr D_\psi|\le M\Tr D_q.
\]
It remains to compute $\Tr D_q$. Since $q(t)=t^2/2$,
\begin{align*}
\Tr D_q
&=\frac12\Tr\bigl(PB^2P-PBPBP\bigr) \\
&=\frac12\Tr\bigl(PB(I-P)BP\bigr)
 =\frac12\|PB(I-P)\|_{\mathfrak{S}_2}^2.
\end{align*}
This proves the theorem.
\end{proof}

\medskip
\noindent
Let $A$ be a selfadjoint, semibounded from below operator in a Hilbert space $\mathcal H$ with a domain $\mathcal{D}(A)$, and let $E$ be its spectral measure. Denote 
$$
P_\lambda\ =\ E\left((-\infty,\lambda)\right)\,,\qquad
P_{\mu,\lambda}\ =\ P_\lambda -P_\mu\ =\ E\left([\mu,\lambda)\right)\,,\quad \mu\le\lambda\,.
$$ 
We shall now obtain an estimate of the Hilbert--Schmidt norm appearing 
in the right hand side of \eqref{DoubSide} with $P=P_\lambda$ and $I-P=I-P_\lambda=P_{\lambda,\infty}$. Without loss of generality we assume that $A>0$; this can always be achieved by adding a sufficiently large positive constant to $A$.

\begin{theorem}\label{1.5}
Let the operator $B$ be the same as in Theorem \ref{DSide}. Assume that $A>0$ 
and $P_\lambda BA\in {\mathfrak{S}_2}$.
Then for all $\psi\in W^{2,\infty}(\mathcal{I})$ and $\lambda> \mu>0$ we have
    \begin{align*}
\|P_{\lambda-\mu}BP_{\lambda,\infty}\|_{\mathfrak{S}_2}
\le \|(A-\lambda)^{-1}P_{\lambda-\mu}[A,B]\|_{\mathfrak{S}_2}, 
    \end{align*}
and
\begin{multline}\label{Db_S_Est}
\left| {\rm Tr}\, \left( P_\lambda \psi(B) P_\lambda -  P_\lambda\psi(P_\lambda B P_\lambda) P_\lambda \right)\right|\\
\le 
 \frac12\,\|\psi''\|_{L^\infty(\mathcal{I})}\, \left(\|P_{\lambda-\mu,\lambda}B\|_{\mathfrak{S}_2}^2
\ +\ \|(A-\lambda)^{-1}P_{\lambda-\mu}[A,B]\|_{\mathfrak{S}_2}^2\right).
\end{multline} 
\end{theorem}

\noindent
For the proof of Theorem \ref{1.5} we need the following statement from \cite{Peller}.
\begin{lemma}\label{double_operator_integral_estimate}
    If $E$ is the spectral measure of $A$ and $T\in\mathfrak{S}_2$, then for every bounded measurable $f$ on $\mathbb{R}^2$ one has
\[
\left\|\iint f(t,s)\,dE(t)\,T\,dE(s)\right\|_{\mathfrak{S}_2}\le \|f\|_{L^\infty(\mathbb{R}^2)}\,\|T\|_{\mathfrak{S}_2}.
\]
\end{lemma}

\begin{proof}[Proof of Theorem 2.2]
    Applying Theorem \ref{DSide} with $P=P_\lambda$ we only need to estimate
\[
\|P_\lambda BP_{\lambda,\infty}\|_{\mathfrak{S}_2}^2.
\]
By orthogonality of the spectral projections,
\[
\|P_\lambda BP_{\lambda,\infty}\|_{\mathfrak{S}_2}^2
=\|P_{\lambda-\mu,\lambda}BP_{\lambda,\infty}\|_{\mathfrak{S}_2}^2
 +\|P_{\lambda-
\mu}BP_{\lambda,\infty}\|_{\mathfrak{S}_2}^2.
\]
The first term is bounded by $\|P_{\lambda-\mu,\lambda}B\|_{\mathfrak{S}_2}^2$, so it remains to estimate the second one.

\noindent
Let
\[
\chi(t)=\mathbf 1_{(-\infty,\lambda)}(t),
\qquad
\chi_\mu(t)=\mathbf 1_{(-\infty,\lambda-\mu)}(t).
\]
Since
\[
P_{\lambda-\mu}[A,B]P_{\lambda,\infty}
=\iint \chi_\mu(t)(1-\chi(s))(t-s)\,dE(t)B\,dE(s),
\]
we have
\[
(A-\lambda)^{-1}P_{\lambda-\mu}[A,B]P_{\lambda,\infty}
=\iint \chi_\mu(t)(1-\chi(s))\frac{t-s}{t-\lambda}\,dE(t)B\,dE(s).
\]
Now define
\[
f_\mu(t,s)=\chi_\mu(t)(1-\chi(s))\frac{t-\lambda}{t-s}.
\]
On the support of $f_\mu$ one has $t<\lambda-\mu<\lambda\le s$, hence $0\le f_\mu(t,s)\le1$. Therefore,
\begin{align*}
P_{\lambda-\mu}BP_{\lambda,\infty}
&=\iint f_\mu(t,s)\,dE(t)
\bigl((A-\lambda)^{-1}P_{\lambda-\mu}[A,B]P_{\lambda,\infty}\bigr)
\,dE(s),
\end{align*}
so Lemma $\ref{double_operator_integral_estimate}$ yields
\begin{multline}\nonumber
\|P_{\lambda-\mu}BP_{\lambda,\infty}\|_{\mathfrak{S}_2}
\le \|(A-\lambda)^{-1}P_{\lambda-\mu}[A,B]P_{\lambda,\infty}\|_{\mathfrak{S}_2}\\
\le \|(A-\lambda)^{-1}P_{\lambda-\mu}[A,B]\|_{\mathfrak{S}_2}.
\end{multline}
Substituting this bound into the estimate from Theorem \ref{DSide} proves the claim.
\end{proof}
%


\section{Proofs of the main results }

\subsection{ Proof of Theorem \ref{SBLY}}
Let ${P}_\lambda$ be the orthogonal projection  whose kernel $e(x,y,\lambda)$ is defined in 
\eqref{SpF}
\begin{equation*} 
e(x,y,\lambda) =(2\pi)^{-1} \sum_n \underset{|\xi_3|(2n+1)<\lambda}\int \Pi_{\xi_3,n}(x',y') e^{i\xi_3(x_3-y_3)} d\xi_3.
\end{equation*}
The proof reduces to applying the Berezin-Lieb inequality. The operator $P_\lambda B\in \mathfrak{S}_2$ and therefore is compact.
Then by applying Proposition \ref{B-L} 
we find 
$$
{\rm Tr}\, \varphi (P_\lambda BP_\lambda)  \le  {\rm Tr}\, P_\lambda \varphi (B) P_\lambda.
$$
For the right hand side in the latter inequality we use  Proposition \ref{prop1} 
that holds even for functions that are not necessarily convex and obtain
$$
 {\rm Tr}\, P_\lambda \varphi (B) P_\lambda = \frac{\lambda^2}{32} \, \int_{\mathbb {R}^3} \varphi (b(x)) dx.
$$
The proof is complete.

\subsection{ Proof of Theorem \ref{Szego}}

In order to prove Theorem \ref{Szego} we need the following statements. 

\medskip
\noindent
For a function $g$, denote by $M_g$ the multiplication operator, and $K_g$ its kernel.

\begin{lemma}\label{first_term}
    Let $m:[0,\infty)\to\mathbb{C}$ be bounded and compactly supported. Then for every $g\in L^2(\mathbb{R}^3)$,
    $$\|M_gm(A)\|^2_{\mathfrak{G}_2}=\frac{1}{16}\left(\int_0^\infty |m(\eta)|^2\eta d\eta\right)\|g\|^2_{L^2(\mathbb{R}^3)}.$$
\end{lemma}
\begin{proof}
    The kernel of $m(A)$ is given by
    $$K_m(x,y)=\frac{1}{2\pi}\sum_{n=0}^\infty\int_{\mathbb{R}}m((2n+1)|\xi_3|)\Pi'_{\xi_3,n}(x',y')e^{i\xi_3(x_3-y_3)}d\xi_3.$$
    Therefore,
    \begin{align}\label{eqt_1}
    \|M_gm(A)\|^2_{\mathfrak{G}_2}=\int_{\mathbb{R}^3}|g(x)|^2\left(\int_{\mathbb{R}^3}|K_m(x,y)|^2dy\right)dx.
    \end{align}
    The inner integral is the diagonal kernel of $m(A)m(A)^*=|m|^2(A)$. Hence
    \begin{align*}
        K_{|m|^2(A)}(x,x)&=\frac{1}{2\pi}\sum_{n=0}^\infty \int_{\mathbb{R}}|m((2n+1)|\xi_3|)|^2\Pi'_{\xi_3,n}(x',x')d\xi_3\\
        &=\frac{1}{4\pi^2}\sum_{n=0}^\infty\int_{\mathbb{R}}|m((2n+1)|\xi_3|)|^2|\xi_3|d\xi_3\\
        &=\frac{1}{2\pi^2}\sum_{n=0}^\infty\frac{1}{(2n+1)^2}\int^\infty_0|m(\eta)|^2\eta d\eta,\text{ with }\eta=(2n+1)|\xi_3|,\\
        &=\frac{1}{16}\int^\infty_0|m(\eta)|^2\eta d\eta.
    \end{align*}
    Substituting this into (\ref{eqt_1}) gives us the desired identity.
\end{proof}

\begin{lemma}\label{second_term}
    Let $m$ be defined as in Lemma \ref{first_term}, then for arbitrary $g_1,g_2\in L^2(\mathbb{R}^3)$,
    $$\sum_{j=1}^2\left\|M_{g_j}X_jm(A)\right\|^2_{\mathfrak{G}_2}\leq\frac{1}{16}\left(\int_0^\infty \eta^2|m(\eta)|^2d\eta\right)\sum_{j=1}^2\|g_j\|^2_{L^2(\mathbb{R}^3)}.$$
\end{lemma}
\begin{proof}
    Note first
    $$\left\|M_{g_j}X_jm(A)\right\|^2_{\mathfrak{G}_2}=\int_{\mathbb{R}^3}|g_j(x)|^2\left(\int_{\mathbb{R}^3}\left|X_jK_m(x,y)\right|^2dy\right)dx.$$

 By Plancherel's theorem in $y_3$ and denoting by $\widehat{X}_j(\xi_3)$ the Fourier transform of $X_j$ in $y_3$, we have
    \begin{align}
        \int_{\mathbb{R}^3}&\left|X_jK_m(x,y)\right|^2dy \notag\\
        &=\frac{1}{2\pi}\int_\mathbb{R}\int_{\mathbb{R}^2}\left|\sum_{n=0}^\infty m(\mu_n(\xi_3))(\widehat{X}_j(\xi_3)\Pi'_{\xi_3,n})(x',y')\right|^2dy'd\xi_3 \notag\\
        &=\frac{1}{2\pi}\sum_{n=0}^\infty \int_{\mathbb{R}}|m(\mu_n(\xi_3))|^2K_{\widehat{X}_j\Pi'_{\xi_3,n}\widehat{X}_j}(x',x')d\xi_3, \label{eqt_4}
    \end{align}
    where (\ref{eqt_4}) is obtained by the orthogonality of Landau projections. For $\xi_3\neq 0$, set $s=\sgn \xi_3$ and
$$
    a_{\xi_3}=\frac{\widehat{X}_1(\xi_3)+is\widehat{X}_2(\xi_3)}{\sqrt{2|\xi_3|}},\quad 
    a_{\xi_3}^*=\frac{\widehat{X}_1(\xi_3)-is\widehat{X}_2(\xi_3)}{\sqrt{2|\xi_3|}}.
$$
Since $[\widehat{X}_1(\xi_3),\widehat{X}_2(\xi_3)]=i\xi_3,$
$$
[a_{\xi_3},a_{\xi_3}^*]=I,\quad \widehat{X}_1(\xi_3)^2+\widehat{X}_2(\xi_3)^2=|\xi_3|(2a_{\xi_3}^*a_{\xi_3}+1).
$$
Set $N_{\xi_3}=a_{\xi_3}^*a_{\xi_3}$, then $\Pi'_{\xi_3,n}$ is the spectral projection of $N_{\xi_3}$ with eigenvalue $n$. Moreover, since
$$
[N_{\xi_3},a_{\xi_3}]=-a_{\xi_3}\quad [N_{\xi_3},a_{\xi_3}^*]=a_{\xi_3}^*,
$$
it follows that
\begin{align}\label{eqt_2}
a_{\xi_3}\Pi'_{\xi_3,n}a_{\xi_3}^*=n\Pi'_{\xi_3,n-1},\quad a_{\xi_3}^*\Pi'_{\xi_3,n}a_{\xi_3}=(n+1)\Pi'_{\xi_3,n+1},
\end{align}
with $\Pi'_{\xi_3,-1}=0$. Solving for $\widehat{X}_1,\widehat{X}_2$ in terms of $a_{\xi_3},a_{\xi_3}^*$ gives
\begin{align}\label{eqt_3}
\widehat{X}_1=\sqrt{\frac{|\xi_3|}{2}}(a_{\xi_3}+a^*_{\xi_3}),\quad \widehat{X}_2=is\sqrt{\frac{|\xi_3|}{2}}(a_{\xi_3}^*-a_{\xi_3}).
 \end{align}
Combining (\ref{eqt_2}) and (\ref{eqt_3}) gives
 \begin{align*}
\widehat{X}_1\Pi'_{\xi_3,n}\widehat{X}_1+\widehat{X}_2\Pi'_{\xi_3,n}\widehat{X}_2=|\xi_3|\left(n\Pi'_{\xi_3,n-1}+(n+1)\Pi'_{\xi_3,n+1}\right).
\end{align*}
Taking diagonal kernels and using $\Pi'_{\xi_3,n}(x',x')=|\xi_3|/2\pi$,
\begin{align}\label{eqt_5}
 K_{\widehat{X}_1\Pi'_{\xi_3,n}\widehat{X}_1+\widehat{X}_2\Pi'_{\xi_3,n}\widehat{X}_2}(x',x')=(2n+1)\frac{|\xi_3|^2}{2\pi}.
 \end{align}
Using (\ref{eqt_4}) and (\ref{eqt_5}) we have
\begin{align*}
        \sum_{j=1}^2\int_{\mathbb{R}^3}\left|X_jK_m(x,y)\right|^2dy&=\frac{1}{4\pi^2}\sum_{n=0}^\infty\int_{\mathbb{R}}(2n+1)|\xi_3|^2|m((2n+1)|\xi_3|)|^2d\xi_3\\
        &=\frac{1}{16}\int_0^\infty \eta^2|m(\eta)|^2d\eta.
\end{align*}
Consequently,
\begin{align*}
\sum_{j=1}^2\|M_{g_j}X_jm(A)\|^2_{\mathfrak{G}_2}&\leq \int_{\mathbb{R}^3}\left(\sum_{j=1}^2|g_j(x)|^2\right)\left(\sum_{j=1}^2\int_{\mathbb{R}^3}\left|X_jK_m(x,y)\right|^2dy\right)dx\\
&\leq \frac{1}{16}\left(\int_0^\infty \eta^2|m(\eta)|^2d\eta\right)\sum_{j=1}^2\|g_j\|^2_{L^2(\mathbb{R}^3)}.
\end{align*}
This proves the lemma.
\end{proof}

\begin{proposition}\label{PBA_HS}
Let $b\in C^2_0(\mathbb{R}^3)$. Then $P_\lambda BA\in\mathfrak{G}_2$.
\end{proposition}

\begin{proof}
Note that $P_\lambda BA = P_\lambda A B -  P_\lambda [A,B]$ where 
\begin{equation}\label{[AB]}
[A,B] = 2M_{X_1 b}X_1 + 2M_{X_2 b} X_2+ M_{Ab}.
\end{equation}
Clearly we have $P_\lambda M_{Ab}\in\mathfrak{G}_2$. 
Let us first consider the operator $P_\lambda A B \in\mathfrak{G}_2$. Indeed
\begin{multline}\label{PAB}
\|P_\lambda A B\|^2_{\mathfrak{G}_2} = 
\frac{1}{2\pi}\sum_{n}\int_{\Bbb R^3}\int_{ (2n+1)|\xi| < \lambda } |(2n+1)\xi_3|^2 \Pi'_{\xi_3,n}(x',x')  b^2(x) d\xi_3 dx\\
\le 
\frac{\lambda^2}{4\pi^2} \sum_{n} \int_{ (2n+1)|\xi| < \lambda } |\xi| \, d\xi \int_{\Bbb R^3} b^2(x) dx
\le \frac{\lambda^4}{64} \int_{\Bbb R^3} b^2(x) dx.
\end{multline}
Note that in the expansion  
$$ 
P_\lambda M_{X_j b} X_j = P_\lambda X_j M_{X_j b} - P_\lambda M_{X^2_jb}.
$$
the operators $M_{X^2_jb}$, $j=1,2$,  are multiplication operators by a function from $C_0(\Bbb R^3)$ and thus   $ P_\lambda M_{X^2_jb} \in\mathfrak{G}_2 $.
For the proof $P_\lambda X_j M_{X_j b}\in\mathfrak{G}_2$ we use similar arguments as in \eqref{PAB} and the equations \eqref{eqt_3}.
\end{proof}

\begin{proposition}\label{off_diag}
    For $b\in C^2_0(\mathbb{R}^3)$, we have
    $$\|P_{\lambda}B(I-P_\lambda)\|_{\mathfrak{G}_2}^2=O(\lambda^{3/2}),\quad\lambda\to\infty.$$
\end{proposition}
\begin{proof}
Let $0<\mu<\lambda$. Orthogonality of spectral projections gives
\begin{align}
        \|P_{\lambda}B(I-P_\lambda)\|_{\mathfrak{G}_2}^2&=\|P_{\lambda-\mu,\lambda}B(I-P_\lambda)\|_{\mathfrak{G}_2}^2+\|P_{\lambda-\mu}B(I-P_\lambda)\|_{\mathfrak{G}_2}^2\notag\\
    &\leq \|P_{\lambda-\mu,\lambda}B\|^2_{\mathfrak{G}_2}+\|P_{\lambda-\mu}B(I-P_{\lambda})\|^2_{\mathfrak{G}_2}. \label{estimate}
\end{align}
For the first term of (\ref{estimate}), since $B=B^*$,
$$
\|P_{\lambda-\mu,\lambda}B\|^2_{\mathfrak{G}_2}=\|BP_{\lambda-\mu,\lambda}\|^2_{\mathfrak{G}_2}.
$$
    Applying Lemma \ref{first_term} with $g=b$ and $m=\mathbf{1}_{[\lambda-\mu,\lambda)}$, we have
\begin{align}\label{first_estimate}
        \|P_{\lambda-\mu,\lambda}B\|^2_{\mathfrak{G}_2}=\frac{1}{16}\int_{\lambda-\mu}^\lambda\eta d\eta\|b\|^2_{L^2}=\frac{2\lambda\mu-\mu^2}{32}\|b\|_{L^2}^2=O(\lambda\mu).
 \end{align}
 By Proposition \ref{PBA_HS},  $P_\lambda BA\in\mathfrak{G}_2$ and thus we can 
 apply Theorem \ref{1.5}  for the second term of (\ref{estimate}) and obtain 
\begin{align}\label{sec_main}
        \|P_{\lambda-\mu}B(I-P_\lambda)\|_{\mathfrak{G}_2}\leq \|(A-\lambda)^{-1}P_{\lambda-\mu}[A,B]\|_{\mathfrak{G}_2}.
\end{align}
Set
$$
m_\mu(\eta)=\frac{\mathbf{1}_{[0,\lambda-\mu)}(\eta)}{\eta-\lambda}.
$$
 Then $m_\mu(A)=(A-\lambda)^{-1}P_{\lambda-\mu}$. Since $m_\mu(A)=m_\mu(A)^*$ and $[A,B]^*=-[A,B]$,
$$
\|m_\mu(A)[A,B]\|_{\mathfrak{G}_2}
    =\|[A,B]m_\mu(A)\|_{\mathfrak{G}_2}.
 $$
 In \eqref{[AB]} we have noticed the commutator identity on \(C_0^2(\mathbb{R}^3)\) 
$$
[A,B]=2M_{X_1b}X_1+2M_{X_2b}X_2+M_{Ab}.
$$
By the triangle inequality and Cauchy-Schwarz,
\begin{align}\label{sec_sub}
        \|m_\mu(A)[A,B]\|^2_{\mathfrak{G}_2}\leq C \left(\sum_{j=1}^2\left\|M_{X_jb}X_jm_\mu(A)\right\|^2_{\mathfrak{G}_2}+\left\|M_{Ab}m_\mu(A)\right\|^2_{\mathfrak{G}_2} \right),
    \end{align}
    where $C$ is a constant. Now Lemma \ref{first_term} gives
    \begin{align}
        \left\|M_{Ab}m_\mu(A)\right\|^2_{\mathfrak{G}_2}&=\frac{1}{16}\int_0^{\lambda-\mu}\frac{\eta}{(\lambda-\eta)^2}d\eta\|Ab\|^2_{L^2}\notag\\
        &=\frac{1}{16}\left(\frac{\lambda}{\mu}-1-\log\frac{\lambda}{\mu}\right)\|Ab\|^2_{L^2}=O(\lambda/\mu). \label{first_es_term}
\end{align}
Lemma \ref{second_term} gives
\begin{align}
        \sum_{j=1}^2\left\|M_{X_jb}X_jm_\mu(A)\right\|^2_{\mathfrak{G}_2}&\leq \frac{1}{16}\int_0^{\lambda-\mu}\frac{\eta^2}{(\lambda-\eta)^2}d\eta\sum_{j=1}^2\|X_jb\|^2_{L^2}\notag\\
        &=\frac{1}{16}\left(\frac{\lambda^2}{\mu}-2\lambda\log\frac
        {\lambda}{\mu}-\mu\right)\sum_{j=1}^2\|X_jb\|^2_{L^2}=O(\lambda^2/\mu). \label{second_est_term}
    \end{align}
Here $b\in C^2_0(\mathbb{R}^3)$ ensures $X_jb,Ab\in L^2(\mathbb{R}^3)$. Since $\lambda/\mu\leq\lambda^2/\mu$ for $\lambda\geq 1$, plugging (\ref{first_es_term}) and (\ref{second_est_term}) back into (\ref{sec_sub}), together with (\ref{sec_main}) we obtain
\begin{align}\label{second_estimate}
        \|P_{\lambda-\mu}B(I-P_\lambda)\|_{\mathfrak{G}_2}^2\leq C_b\frac{\lambda^2}{\mu}.
\end{align}
Combining (\ref{first_estimate}) and (\ref{second_estimate}),
$$
\|P_\lambda B(I-P_\lambda)\|_{\mathfrak{G}_2}^2\leq C_b\left(\lambda\mu+\frac{\lambda^2}{\mu}\right),\quad 0<\mu<\lambda.
$$
For $\lambda> 1$, choosing $\mu=\sqrt{\lambda}$ gives the desired result.
\end{proof}

\smallskip
\noindent
We now complete the proof of Theorem \ref{Szego} by applying Propositions \ref{off_diag}, \ref{prop1}  and 
Theorem \ref{DSide}.

\bigskip

\bibliographystyle{amsplain}

\end{document}